\documentclass[11pt]{article}

\usepackage{amssymb,amsmath,amsthm}
\usepackage{enumerate}
\usepackage{epsfig}
\usepackage{url}
\usepackage{color}
\usepackage{graphicx}
\usepackage{soul}
\usepackage{lineno}
\usepackage{latexsym}

\usepackage{epsfig, verbatim}

\theoremstyle{plain}
\newtheorem{theorem}{Theorem}
\newtheorem{lemma}[theorem]{Lemma}
\newtheorem{corollary}[theorem]{Corollary}

\newtheorem{claim}{Claim}[theorem]

\theoremstyle{definition}

\newtheorem{problem}[theorem]{Problem}
\newtheorem{conjecture}[theorem]{Conjecture}

\def\pro{\noindent {\bf Proof. }}
\def\fpro{\hfill $\square$\\\vspace{2mm}}

\begin{document}

\title {\bf Colorful paths for 3-chromatic graphs}

\author{St\'ephane Bessy and Nicolas Bousquet\\~\\ Universit\'e
  Montpellier 2 - CNRS, LIRMM,\\ 161 rue Ada, 34392 Montpellier Cedex 5,
  France,\\ {\tt bessy,bousquet@lirmm.fr}} 
 \date{}

\maketitle

\begin{abstract}
In this paper, we prove that every 3-chromatic connected graph, except
$C_7$, admits a 3-vertex coloring in which every vertex is the
beginning of a 3-chromatic path. It is a special case of a conjecture
due to S.~Akbari, F.~Khaghanpoor, and S.~Moazzeni, cited in
[P.J. Cameron, Research problems from the BCC22, \emph{Discrete Math.}
  {\bf 311} (2011), 1074--1083], stating that every connected graph
$G$ other than $C_7$ admits a $\chi (G)$-coloring such that every
vertex of $G$ is the beginning of a colorful path (i.e. a path of on
$\chi(G)$ vertices containing a vertex of each color). We also provide
some support for the conjecture in the case of 4-chromatic graphs.
\end{abstract}

~


~

\section{Introduction}

In this paper, we deal with oriented and non-oriented graphs. When it
is not specified, graphs are supposed to be non-oriented.  Notations
not given here are consistent with~\cite{BM08}.  The vertex set of a
graph or an oriented graph $G$ is denoted by $V(G)$ and its edge set
(or arc set) by $E(G)$\footnote{Throughout the paper, we use the
  notation $xy$ to indicate the (oriented) arc from $x$ to $y$, while
  $\{x,y\}$ designates the (non-oriented) edge between $x$ and $y$.}
Classically, for a vertex $x$ of a graph $G$, a vertex $y$ with
$\{x,y\}\in E(G)$ is called a \emph{neighbour} of $x$.  The set of all
the neighbours of $x$, denoted by $N_G(x)$, is the
\emph{neighbourhood} of $x$ in $G$. In the oriented case, an
\emph{out-neighbour} (resp. \emph{in-neighbour}) of a vertex $x$ of an
oriented graph $G$ is a vertex $y$ with $xy\in E(G)$ (resp. $yx\in
E(G)$). Similarly, the set of all the out-neighbours
(resp. in-neighbours) of $x$ in $G$, denoted by $N^+_G(x)$
(resp. $N^-_G(x)$) is the \emph{out-neighbourhood}
(resp. \emph{out-neighbourhood}) of $x$ in $G$.\\ In a graph $G$, we
denote by $x_1\dots x_{\ell+1}$ the \emph{path} of length $\ell$ on
the distinct vertices $\{x_1,\dots ,x_{\ell+1}\}$ with edges
$\{x_1,x_2\}, \{x_2,x_3\}, \dots , \{x_\ell,x_{\ell+1}\}$. We denote
also by $x_1 \ldots x_\ell x_1$ the \emph{cycle} $C_{\ell}$ of length
$\ell$ on the distinct vertices $\{x_1, \ldots ,x_{\ell}\}$ with edges
 $\{x_1,x_2\}, \ldots ,\{x_{\ell-1},x_{\ell}\},
\{x_{\ell},x_1\}$. Classically, these notions are extended to oriented
graphs, where the arcs $x_ix_{i+1}$ replace the edges
$\{x_i,x_{i+1}\}$ (computed modulo $\ell$ for the oriented cycle
$C_\ell$).\\
 
A \emph{$k$-(proper) coloring} of a graph $G$ is a mapping $c:
V(G)\rightarrow \{1,\dots ,k\}$ such that $c(u)\neq c(v)$ if $u$ and
$v$ are adjacent in $G$.  The \emph{chromatic number} of $G$, denoted
by $\chi (G)$, is the smallest integer $k$ for which $G$ admits a
$k$-coloring and thus, we say that $G$ is a \emph{$\chi (G)$-chromatic
  graph}. For a $k$-coloring of a graph $G$, a \emph{rainbow path} of
$G$ is a path whose vertices have all distinct
colors. Given a $\chi (G)$-coloring of $G$, a rainbow
  path on $\chi (G)$ vertices is a \emph{colorful path}. In particular
  a rainbow path is transversal to the set of colors (\emph{i.e.} it has a non
  empty intersection with every color class). Finding structures
  transversal to a partition of the ground set is a general problem
  in combinatorics. Examples arise from Steiner Triple
  Systems (see~\cite{FuSi05}), systems of representatives
  (see~\cite{ABZ07}) or extremal graph theory (see~\cite{KeMy13}).  
  Rainbow and colorful paths have been extensively studied in the last few years,
  see for instance~\cite{AKM}, \cite{ALN11}, \cite{Fung89}, \cite{Li01}
  and~\cite{Lin07}.  In this paper, we concentrate on a conjecture of
  S.~Akbari, F.~Khaghanpoor and S.~Moazzeni raised in~\cite{AKM}
  (also cited in~\cite{BCC22}).

\begin{conjecture}[S.~Akbari, F.~Khaghanpoor and S.~Moazzeni~\cite{AKM}]
\label{conj:colorful}
Every connected graph $G$ other than $C_7$ admits a $\chi (G)$-coloring
such that every vertex of $G$ is the beginning of a colorful path.
\end{conjecture}

Conjecture~\ref{conj:colorful} holds for $1$-chromatic graphs and
$2$-chromatic graphs.  Indeed in connected bipartite graphs, every
vertex is connected to a vertex of another color.  The classical proof
of Gallai-Roy Theorem also shows that in any $\chi(G)$-coloring of a
graph $G$, there exists at least one colorful path (see~\cite{BM08},
for instance). Furthermore, much more is known concerning this
conjecture which, through recent, have already received attention.
In~\cite{ALN11}, S.~Akbari, V.~Liaghat, and A.~Nikzad proved that
Conjecture~\ref{conj:colorful} is true for the graphs $G$ having a
complete subgraph of size $\chi (G)$.  They also proved that every
graph $G$ admits a $\chi (G)$-coloring such that every vertex is the
beginning of a rainbow path on $\lfloor \frac{\chi (G)}{2} \rfloor$
vertices. This result was improved by M. Alishahi, A. Taherkhani and
C. Thomassen in~\cite{ATT11}, who showed that we can obtain rainbow
paths on $\chi(G)-1$ vertices.

In this paper, we give another evidence for
Conjecture~\ref{conj:colorful}, and prove it for $3$-chromatic graphs.

\begin{theorem}
\label{theo:3colorful}
Every connected 3-chromatic graph $G$ other than $C_7$ admits a
3-coloring such that every vertex of $G$ is the beginning of a
colorful path.
\end{theorem}

The proof of Theorem~\ref{theo:3colorful} uses an auxiliary oriented
graph build from a coloring of the instance graph.  This oriented
graph was already used in~\cite{ALN11}. In the next section, we recall
its definition and strengthen the results known about it to obtain
some useful lemmas. In Section~\ref{sec:proof}, we use these tools to
derive the proof of Theorem~\ref{theo:3colorful}.  Finally, in
Section~\ref{sec:conclusion}, we conclude the paper with some remarks
and open questions.  In particular, we prove that
Conjecture~\ref{conj:colorful} is true for 4-chromatic graphs
containing a cycle of length four.

\section{Preliminaries}

In this section, $G=(V,E)$ is a connected graph and $c$ is a proper
coloring of $G$ with $\chi (G)$ colors. Here, $G$ is not
necessarily 3-chromatic and, for short, we write $\chi$ instead of
$\chi(G)$.  In the following, we will consider modifications of
colors and all these modifications have to be understood modulo
$\chi$.\\ As defined in~\cite{ALN11}, the oriented graph $D_c$ has
vertex set $V$ and $ab$ is an arc of $D_c$ if $\{a,b\}$ is an edge of
$G$ and the color of $b$ equals the color of $a$ plus one 
(this oriented graph was first introduced in~\cite{Guichard, Zhu01}).
A colorful path starting at the vertex $x$ is called a \emph{certifying
  path for $x$}. A colorful path $x_1\dots  x_\chi$ is \emph{forward}
(resp. \emph{backward}) if for every $i\in \{1,\dots ,\chi-1\}$ we have
$c(x_{i+1})=c(x_i)+1 \mod \chi$ (resp.  $c(x_{i+1})=c(x_i)-1 \mod
\chi$). Note that a forward (resp. backward) certifying path for a
vertex $x$ is an oriented path in $D_c$ on $\chi$ vertices starting
(resp. ending) at $x$.

An \emph{initial section} of $D_c$ is a subset $X$ of $V$ such that
there is no arc of $D_c$ entering into $X$ (i.e. from $V(G)\setminus
X$ to $X$). The \emph{initial recoloring} of $X$ consists of
reducing the color used on each vertex in $X$ by one. We have the following
basic facts (which are mentioned in~\cite{ALN11}, but we recall here
their short proofs for the sake of completeness).

\begin{lemma}[S.~Akbari et al.~\cite{ALN11}]
\label{lem:initialrecoloring}
An initial recoloring of an initial section is still a proper coloring.
\end{lemma}

\pro Let $c$ be a coloring of $G$ and $X$ an initial section of
$D_c$.  We denote by $c'$ the coloring of $G$ obtained after the
initial recoloring of $X$. Let $x$ and $y$ be two adjacent vertices.
If both $x$ and $y$ are not in $X$, we have $c'(x)=c(x)\neq
c(y)=c'(y)$.  If both $x$ and $y$ are in $X$, we have $c'(x)=c(x)-1
\neq c(y)-1 = c'(y)$.  So, by symmetry we may assume that $x \notin X$
and $y \in X$.  Since $X$ is an initial section, there is no arc from
$x$ to $y$ in $D_c$ and then we have $c(x) \neq c(y)-1$. Thus we have
$c'(x)=c(x) \neq c(y)-1 = c'(y)$.  \fpro

\noindent
We will intensively use Lemma~\ref{lem:initialrecoloring} to prove
Theorem~\ref{theo:3colorful}, and so, without refereeing it
precisely.  Notice that when performing an initial recoloring on an
initial section $X$, we remove from $D_c$ all the arcs leaving $X$ and
possibly add some arcs entering into $X$ (the arcs $xy$ with
$\{x,y\}\in E(G)$, $x\notin X$, $y \in X$ and $c(x)=c(y)-2$). 
Moreover, we do not create any arc leaving $X$. Indeed suppose by contradiction
that an arc $xy$ is created with $x \in X$ and $y \notin X$,
then in the original coloring $c$, we must have $c(x)=c(y)$, 
contradicting $c$ being proper.
The other arcs, standing inside or outside $X$ remain unchanged.

Similarly, a subset $X$ of vertices is a \emph{terminal section} of
$D_c$ if there is no arc leaving $X$ (i.e. from $X$ to $V(G)\setminus
X$). The \emph{terminal recoloring} of $X$ consists in adding one to
the color of the vertices of $X$. As for the initial recoloring,
this coloring is still proper. Note also that, when performing a
terminal recoloring of $X$, we remove from $D_c$ all the arcs entering
into $X$ and possibly add some arcs leaving $X$ (the arcs $xy$ with
$\{x,y\}\in E(G)$, $x\in X$, $y \notin X$ and $c(x)=c(y)-2$).\\ Using
initial and terminal recolorings, we prove some basic facts on the
existence of colorful paths.  Two colorings $c$ and $c'$ are
\emph{identical} on $X$ if $c(x)=c'(x)$ for all $x \in X$.

\begin{lemma}[S.~Akbari et al.~\cite{ALN11}]
\label{lem:reachable-set}
Let $c$ be a $\chi$-coloring of $G$ and $X$ be a subset of vertices
of $G$.  There exists a $\chi$-coloring $c'$ of $G$ identical to $c$
on $X$ such that every vertex is the beginning of an oriented path of
$D_{c'}$ which ends in $X$.
\end{lemma}

\pro Let $c'$ be a $\chi$-coloring of $G$ identical with $c$ on $X$.
We define $Y_{c'}$ as the set of vertices of $G$ which are the
beginning of an oriented path in $D_{c'}$ ending in $X$. 
The path can have length $0$, \emph{i.e.} $X$ is included in $Y_{c'}$.
Now, we choose $c'$ a $\chi$-coloring of $G$ identical with $c$ on $X$ with
an associated set $Y_{c'}$ of maximal cardinality.  Let us prove that
$Y_{c'}=V$. Otherwise, notice that, by definition, $Y_{c'}$ is an
initial section of $D_c$, and so that $V\setminus Y_{c'}$ is a
terminal section of $D_c$. Denote by $c_t$ the terminal recoloring of
$V\setminus Y_{c'}$. As $X\subset Y_{c'}$, $c_t$ is also identical to
$c$ on $X$. Moreover, the arcs from $Y_{c'}$ to $V(G)\setminus Y_{c'}$
of $D_{c'}$ are not anymore in $D_{c_t}$ and the only arcs which can
be created are arcs from $V(G)\setminus Y_{c'}$ to $Y_{c'}$ in
$D_{c_t}$.  If no arc from $V(G)\setminus Y_{c'}$ to $Y_{c'}$ is
created, we can repeat the terminal recoloring of $V\setminus Y_{c'}$
until such an arc appears.  As $G$ is connected, the process must stop
at some step, and at least one arc $zz'$ must appear from
$V(G)\setminus Y_{c'}$ to $Y_{c'}$. So, $c_t$ is identical to $c$ on
$X$, and we have $Y_{c'} \cup \{z \} \subseteq Y_{c_t}$ which
contradicts the maximality of $Y_{c'}$. \fpro

In particular, Lemma~\ref{lem:reachable-set} implies that
Conjecture~\ref{conj:colorful} holds if $D_c$ contains an oriented
cycle. Indeed, if $C$ is such a cycle, then $C$ has length a multiple of
$\chi$, and so $C$ has length greater or equal than $\chi$. If we
apply Lemma~\ref{lem:reachable-set} with $X=V(C)$, then we obtain a
$\chi$-coloring $c'$ such that every vertex of $G$ is the beginning
of an oriented path of $D_{c'}$ ending in $V(C)$. So, extending
possibly these paths with some vertices and arcs of $C$, every vertex
of $G$ is the beginning of an oriented path of $D_{c'}$ of length
$\chi$. Thus, $G$ satisfies Conjecture~\ref{conj:colorful}.\\ As
mentioned in~\cite{ALN11}, note that if $G$ contains a clique of
size $\chi (G)$, then for every $\chi$-coloring $c$ of $G$, $D_c$
contains an oriented cycle, and so $G$ verifies
Conjecture~\ref{conj:colorful}.\\

So, we have to focus on the case where $D_c$ is an oriented acyclic
graph.  We introduce some notations for this case.  Let $D$ be an
oriented acyclic graph. The \emph{level partition of $D$} is the
unique partition of $V(D)$ into subsets $(V_1,\cdots, V_k)$ such that
$V_i$ consists of all sinks of the oriented acyclic graph induced by
$D$ on $V \backslash \cup_{j=1}^{i-1} V_j$. As each $V_i$ is the set
of sinks of an acyclic induced oriented subgraph of $D$, it is in
particular an independent set.  The \emph{height of a vertex} $x$,
denoted by $h_D(x)$, is the index of the level $x$ belongs to in the
level partition of $D$. And, the \emph{height of the partition} is the
maximal height of a vertex (i.e. $k$ in our notations, here).  \\ Now,
we introduce notations for an oriented acyclic graph $D_c$ associated
with a $\chi$-coloring $c$ of $G$. Assuming that $D_c$ is acyclic, we
denote by $(V_1^c,\cdots , V_k^c)$ its level partition. Note that by
construction, if $xy$ is an arc of $D_c$ with $x\in V_i$ and $y\in
V_j$ we have $i>j$ and $i-j=1\mod \chi$.  We define the \emph{height
  of $c$} as the height of this level partition. It is also the number
of vertices in a longest oriented path of $D_c$. We denote it by
$h(c)$, and to shorten notations, we write $h_c(x)$ instead of
$h_{D_c}(x)$ to indicate the height of a vertex in the level partition
of $D_c$.  \\ 
Finally, a $\chi$-coloring $c$ is a \emph{nice coloring}
of $G$ if $D_c$ is an oriented acyclic graph with a unique sink. 
Given a nice coloring $c$, every vertex of $D_c$ is the beginning of an oriented path
which ends at this unique sink.  An \emph{in-branching} is an
orientation of a tree in which every vertex has out-degree $1$, except
one vertex, called \emph{the root} of the in-branching. It is 
well-known that, for a fixed vertex $x$ of a digraph $D$, every vertex is
the beginning of an oriented path ending at $x$ if, and only if, $D$
has a spanning in-branching rooted at $x$ (see~\cite{BM08} Chap.~4 for
instance).
Thus, $c$ is nice if, and only if, $D_c$ has a spanning in-branching.
So, if we apply Lemma~\ref{lem:reachable-set} with a set $X$
containing a unique vertex $v$, we obtain a coloring $c$ where every
vertex is the beggining of a path ending in $v$, or equivalently where
$D_c$ has an in-branching rooted at $v$.  If $v$ is not a sink of
$D_c$, then $D_c$ contains an oriented cycle. Otherwise $D_c$ has a
unique sink $v$, and then $c$ is a nice coloring. Thus we have the
following.

\begin{corollary}
\label{col:unique-sink}
Either $G$ admits a $\chi$-coloring $c$ such that $D_c$ contains an
oriented cycle, or for every vertex $v$ of $G$, there is a nice
$\chi$-coloring of $G$ with $v$ as unique sink.
\end{corollary}

Note that, given a nice coloring $c$ of $G$, the vertices belonging
to a same level of the level partition of $D_c$ receive the same
color by $c$. Indeed, $V_1^c$ only contains the unique sink $r$ of
$D_c$, and, as every vertex in $V_i^c$ has an out-neighbour in $V_{i-1}^c$,
an easy induction shows that $c(x)=c(r)-i+1 \mod \chi$ for every $x\in
V_i$.\\
Now, we can establish the following lower bound on the height of $D_c$,
for a nice coloring $c$ of $G$.

\begin{lemma}
\label{lem:height-min}
Let $c$ be a nice $\chi$-coloring of $G$.
We have $h(c)\ge 2\chi-1$.
\end{lemma}

\pro Assume by contradiction that $h(c)\le 2\chi -2$ for a nice
coloring $c$ of $G$. Denote by $r$ the unique sink of $D_c$ (which
forms the level $V_1^c$), and consider the set $X= V^c_{\chi} \cup
V^c_{\chi +1} \cup \dots \cup V^c_{h(c)}$. This set $X$ is not empty
(otherwise $G$ would have a partition in less then $\chi$ independent
sets) and is an initial section of $D_c$. When performing the initial
recoloring of $X$, the color $c(r)+1$ disappears. Indeed, only the
vertices of $V^c_{\chi}$ used this color before the recoloring, and
no vertex of $X$ uses it after the recoloring.  So, we obtain a
$(\chi-1)$-proper coloring of $G$, a contradiction.  \fpro

As a consequence, in a nice coloring of $G$, there exists a backward
certifying path for the sink of $D_c$. As, by
Lemma~\ref{col:unique-sink}, every vertex of $G$ can be the sink of
$D_c$ for a nice coloring $c$, or we find an oriented cycle in $D_c$,
it means that for every vertex $x$, there exists a coloring of $G$
containing a colorful path with end $x$, what was already proved
in~\cite{Li01}.\\ But, we can be more precise. Let $c$ be a nice
coloring of $G$, we denote by $B_c$ the set of vertices of $G$ which
have no certifying path. We have seen that the sink of $D_c$ is not in
$B_c$. Moreover, in the level partition $(V^c_1,\cdots, V^c_k)$ of $D_c$,
every vertex in $V^c_i$ has an out-neighbour in $V^c_{i-1}$, and then,
every vertex in $V^c_i$ with $i\ge \chi$ has a forward certifying
path. Then, we obtain the following.

\begin{lemma}
\label{lem:Bc}
Let $c$ be a nice coloring of $G$.  We have $B_c\subseteq V^c_2\cup
V^c_3\cup \dots \cup V^c_{\chi -1}$.
\end{lemma}

Now, we pay attention to 3-chromatic graphs.

\section{Colorful paths for 3-chromatic graphs}
\label{sec:proof}

In this section, we focus on the special case $\chi =3$ and prove
Theorem~\ref{theo:3colorful}.  Note that, when
considering a 3-coloring $c$ of a 3-chromatic graph $G$, every edge
of $G$ appears as an arc of the oriented graph $D_c$ (indeed, for any
edge ${x,y}$ of $G$, we have $c(x)-c(y)\in \{-1,1\}$). Furthermore,
when we perform an initial recoloring on an initial section $X$ of
$D_c$, all the arcs leaving $X$ become arcs entering into $X$.

\begin{lemma}[S.~Akbari et al.~\cite{ALN11}]
\label{lem:cycleok}
Conjecture~\ref{conj:colorful} is true for every odd cycle except
$C_7$.
\end{lemma}

\pro For the sake of completeness, we just give the coloring yielding
the result. Let $C=v_0v_1\dots v_kv_k'v_{k-1}'\dots v_2'v_1'v_0$ be an
odd cycle different from $C_7$. We define the 3-coloring $c$ of $C$ by
$c(v_0)=3$, $c(v_i)= c(v'_i)=i \mod 3$ for $1\le i\le k-1$, $c(v_k)=k
\mod 3$ and $c(v'_k)=k+1\mod 3$. Now, it is easy to check that if
$k\neq 3$ (i.e. $C\neq C_7$), then every vertex is the beginning of a
colorful path.\fpro

In the following we assume by contradiction that
Theorem~\ref{theo:3colorful} is not true and consider a minimal
counter-example $G$ (subject to its number of vertices) distinct from
$C_7$.  
The only consequence of the minimal cardinality of $G$ we
use is given by the following claim.

\begin{claim}
\label{claim:no-twin}
The graph $G$ does not contain any twins, that is, there is no two vertices
$x$ and $y$ in $G$ with $N_G(x)=N_G(y)$.
\end{claim}

\pro Assume that $G$ has two vertices $x$ and $y$ with
$N_G(x)=N_G(y)$. First, notice that $G\setminus \{ y \}$ is
3-chromatic, as we can extend every coloring of $G\setminus \{ y \}$
to $G$ by coloring $y$ with the color of $x$.  Now, if $G\setminus \{ y
\}=C_7$, then in the coloring of $G$ given Figure~\ref{fig:c7twin}
every vertex is the beginning of a colorful path.
\begin{figure}[!ht]
\centering
\includegraphics[scale=0.5]{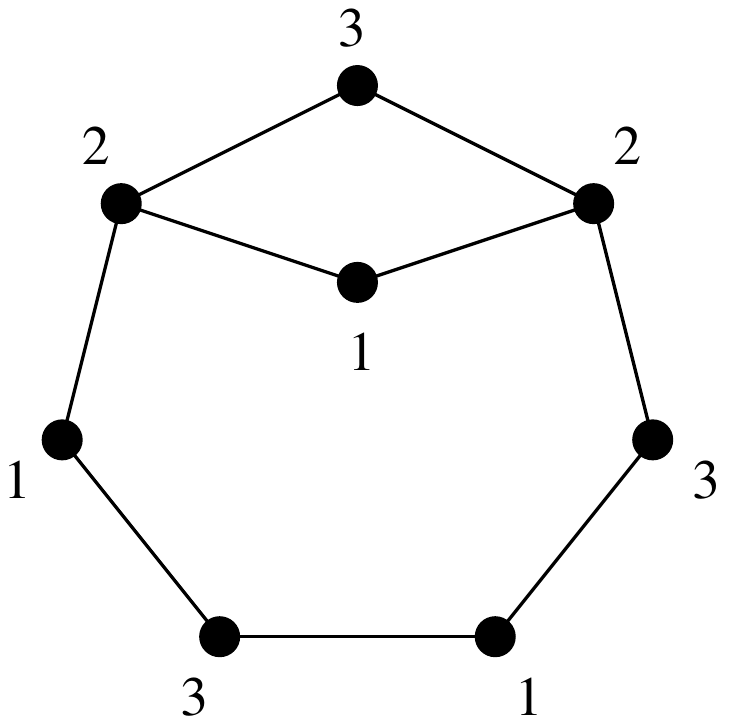}
\caption{A coloring of the 'twinned $C_7$' in which every vertex is the 
beginning of a colorful path.}
\label{fig:c7twin}
\end{figure}
So, $G \setminus \{ y \}$ is different from $C_7$. Since $G$ is a
minimum counterexample, there is a proper coloring $c$ of $G\setminus
\{ y \}$ such that every vertex is the beginning of a colorful
path. Extending the coloring $c$ to $y$ with $c(y)=c(x)$ provides a
certifying path for $y$ since $x$ has one. So, $G$ would not be a
counter-example to Theorem~\ref{theo:3colorful}, a contradiction.
\fpro

Now, let $c$ be a nice coloring of $G$, which exists by
Corollary~\ref{col:unique-sink}. As previously noticed, the associated
oriented graph $D_c$ is acyclic since $G$ is a counter-example to
Theorem~\ref{theo:3colorful}. We can describe precisely the structure of $D_c$
as follows. By Lemma~\ref{lem:Bc}, the set $B_c$ of vertices which do not
have a certifying path is a subset of $V_2$, the second level in the
level partition of $D_c$. So, if we denote by $r_c$ the unique sink of $D_c$,
every vertex of $B_c$ has a unique out-neighbour which is $r_c$.
Moreover, every vertex $b$ of $B_c$ is not the end of an oriented path
of length two in $D_c$. Thus, either $b$ is a source of $D_c$ or all
its in-neighbours are sources of $D_c$, and by construction of $D_c$ these 
in-neighbours belong to levels $V^c_i$ with $i=0 \mod 3$.
Finally, $D_c$ has height at
least five by Lemma~\ref{lem:height-min} and so, at least one vertex
of $V_2$ is the beginning of a backward certifying path. In
particular, we know that $B_c$ is a proper subset of
$V_2$. Figure~\ref{fig:Dc} depicts the situation.\\

\begin{figure}[!ht]
\centering
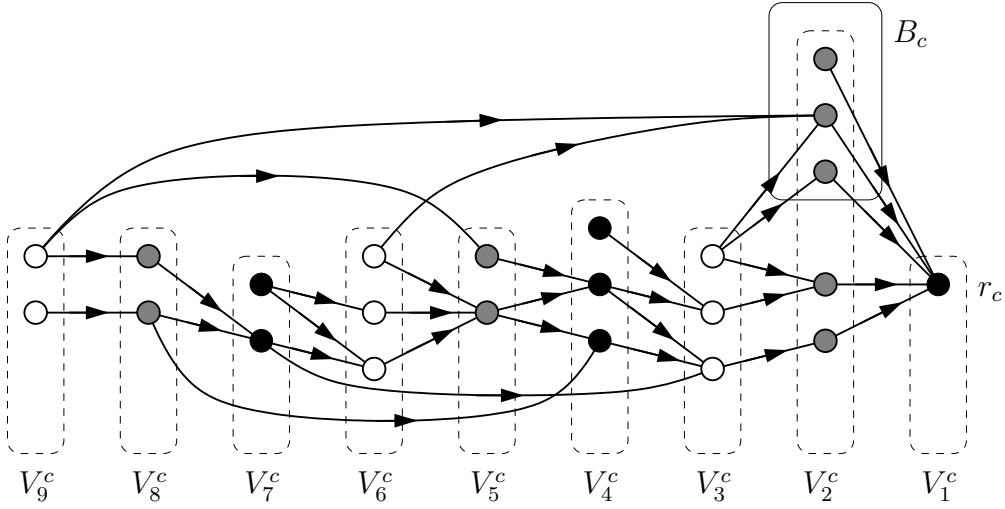
\caption{An illustrative example of oriented graph $D_c$. The vertices
  colored black, gray and white respectively receive value 3, 2 and 1
  by $c$. }
\label{fig:Dc}
\end{figure}

Consider the following special initial recoloring of $D_c$. For a
vertex $b$ of $B_c$, the previous argument ensures that $\{b\}\cup
N_{D_c}^-(b)$ is an initial section of $D_c$. The \emph{switch
  recoloring on $b$} is the initial recoloring on $\{b\} \cup
N_{D_c}^-(b)$. By Lemma~\ref{lem:initialrecoloring}, this coloring is
proper, and moreover, it satisfies the following properties.

\begin{claim}
\label{claim:switch}
Let $b$ be a vertex of $B_c$ and denote by $c'$ the switch recoloring
on $b$. The oriented graph $D_{c'}$ has the following properties:

{\bf (a)} $c'$ is a nice coloring of $G$, the unique sink of
$D_{c'}$ is $b$ and so $V_1^{c'}=\{b\}$.

{\bf (b)} $V^{c'}_2 = \{r_c\}\cup N^-_{D_c}(b)$, $V^{c'}_3= V^c_2
\setminus \{b\}$ and for every $i=4,\dots ,h(c)+1$, if $i=1\mod 3$

\hspace*{0,5cm} we have $V_{i}^{c'} = V_{i-1}^c \setminus N_{D_c}^-(b)$ and if
$i\neq 1\mod 3$ we have $V_{i}^{c'} = V_{i-1}^c$.

{\bf (c)} If $V^c_{h(c)}\setminus N^-_{D_c}(b)\neq \emptyset$ then
$h(c')=h(c)+1$.

{\bf (d)} If $V^c_{h(c)}\subseteq N^-_{D_c}(b)$ then $h(c')=h(c)$.


{\bf (e)} $B_{c'}$ is a subset of $N^-_{D_c}(b)$, i.e. $r_c$ has a
certifying path in $D_{c'}$.
\end{claim}

\pro Denote by $N$ the set $N^-_{D_c}(b)$. The oriented graph $D_{c'}$
is obtained from $D_{c}$ by reversing all the arcs from $\{b\}\cup
N$ to $V\setminus (\{b\}\cup N)$. The vertex $r_c$ is the unique sink of
$D_c[V \setminus (\{b\} \cup N)]$ and $b$ is the unique sink of $D_c[\{b\}
  \cup N]$. Since $r_cb$ is an arc of $D_{c'}$, $b$ is the unique sink
of $D_{c'}$ which proves {\bf (a)}.

Now, we will prove that $V^{c'}_2 = \{r_c\}\cup N$ and $V_{i-1}^c
\setminus (\{b\}\cup N)\subseteq V_{i}^{c'}$ for $i=3,\dots
,h(c)+1$. We call $(\star)$ this property. Assuming that $(\star)$ is
true, each part of the partition $(\{b\}, \{r_c\}\cup N, V_{2}^c
\setminus (\{b\}\cup N), V_{3}^c \setminus (\{b\}\cup N),\dots ,
V_{h(c)}^c\setminus (\{b\}\cup N))$ of $G$ will be respectively
included in the corresponding part of the partition $(V_1^{c'},\dots
,V_{h(c')}^{c'})$.  So, the two partitions will be equal, and using
that $N^-_{D_c}(b)\subseteq \bigcup \{V_i^c : i=0\mod 3 \}$ we have
{\bf (b)}. The proof of $(\star)$ runs by induction on $i$. Let start
with the case $i=2$.  As $r_c$ is the unique out-neighbour of $b$ in
$D_c$, we have $d^-_{D_{c'}}(b)= \{r_c\} \cup N$. In $D_{c'}$, the
vertex $b$ is the unique out-neighbour of $r_c$.  In $D_{c'}$, the
vertex $b$ is also the unique out-neighbour of the vertices of
$N$. Indeed, $N$ is an independent set of $G$ with no in-neighbour in
$D_c$. Moreover, when we perform the switch recoloring on $b$, we
invert all the arc leaving $N$ except the one with head $b$. So, we
have $V_2^{c'}= \{r_c\}\cup N$. Now, assume that for some integer
$i\in \{3, \dots ,h(c)+1\}$ the property $(\star)$ is true for all $j$
with $j< i$. Let $x$ be a vertex of $V_{i-1}^c \setminus (\{b\}\cup
N)$. The out-neighbours of $x$ in $D_c$ are in $\cup_{k=1}^{i-2}
V_{k}^c$. All these vertices are in $\cup_{k=1}^{i-1} V_{k}^{c'}$ by
induction hypothesis. The other possible out-neighbours of $x$ are
vertices of $\{b\} \cup N$ which are in $V_1^{c'} \cup V_2^{c'}$ as
previously shown. Therefore, the height of $x$ is at most $i$ in
$D_{c'}$. Let $y$ be an out-neighbour of $x$ in $D_c$ such that $y \in
V_{i-2}^c$. Since $y$ is not a source in $D_c$ we have $y\notin N$,
and as $y\neq b$ (otherwise, $x$ would be in $N$), we have by
induction hypothesis $y \in V_{i-1}^{c'}$. So, the height of $x$ is
exactly $i$ and $x\in V_{i}^{c'}$, which finally proves $(\star)$ and
{\bf (b)}.

In particular, {\bf (b)} directly implies {\bf (c)}. It also implies {\bf (d)}
easily.  Indeed, assume that $V_{h(c)}^c\subseteq N$, then as
$V_{h(c)}^c\cap N\neq \emptyset$ we have $h(c)=0 \mod 3$.  Thus, by
{\bf (b)}, we have $V_{h(c)+1}^{c'}=V_{h(c)}^c\setminus N =\emptyset$ 
and $V_{h(c)}^{c'}=V_{h(c)-1}^c\neq \emptyset$. So we obtain $h(c')=h(c)$.

To prove {\bf (e)}, as we know that $B_{c'}$ is a subset of
$V^{c'}_2$, which is $\{r_c\}\cup N$ by {\bf (b)}, we just have to
check that $r_c$ is the end of a certifying path in $D_{c'}$. Let
$P=x_4x_3x_2x_1$ be an oriented path of length $3$ in $D_c$ with
$x_1=r_c$ and $x_i\in V_i^c$ for $i=1,2,3,4$. Such a path exists since
the height of $c$ is at least $5$ by Lemma~\ref{lem:height-min}. As
$x_3$ is not a source in $D_c$, we have $x_3\notin N$. As $x_2$ is
certified in $c$ by the path $x_4x_3x_2$, we also have $x_2\neq
b$. So, the oriented path $x_3x_2x_1$ still exists in $D_{c'}$ and is
a certifying path in $D_{c'}$ for $r_c$. So, {\bf (e)} is proved.  \fpro

Before going on the main proof, let us establish the following
technical result.

\begin{claim}
\label{claim:technical}
Let $c$ be a nice coloring of $G$ and assume that there exists at
least one arc from $V^c_{h(c)}$ to $B_c$ in $D_c$. If $X\subseteq
V^c_{h(c)}\cup V^c_{h(c)-1}$ is an initial section of $D_c$, then
every vertex of $D_c\setminus (X\cup B_c)$ has a certifying
path lying in $D_c\setminus X$.
\end{claim}

\pro As $X$ is an initial section of $D_c$, every forward certifying
path starting at a vertex of $D_c\setminus X$ lies in $D_c\setminus
X$.  Thus every vertex of $D_c\setminus (X\cup V_2^c\cup \{r_c\})$ has
a (forward) certifying path in $D_c\setminus X$. We have $X\cap
V_3^c=\emptyset$ since $h(c) \ge 5$ and $X\subseteq V^c_{h(c)}\cup
V^c_{h(c)-1}$. Thus there exists a path on three vertices starting in
$V_3^c$, ending on $r_c$ and lying in $D_c\setminus X$. This path is a
backward certifying path for $r_c$. To conclude, let $z$ be a vertex
of $V_2^c\setminus B_c$. As $z\notin B_c$, $z$ is certified in $D_c$
by a backward path $P=vwz$ on three vertices. By construction we have
$h_c(z)=2\mod 3$ and then $h_c(w)=0\mod 3$ and $h_c(v)=1\mod 3$.  Note
that, since there is an arc from $V^c_{h(c)}$ to $B_c$, we have
$h_c(y)=0\mod 3$ for every vertex $y$ of $V^c_{h(c)}$ and then
$h_c(y')=2\mod 3$ for every vertex $y'$ of $V^c_{h(c)-1}$. Since
$h_c(w)=0$, if $w \in X$ then $w$ must be in $V^c_{h(c)}$, which is
impossible since $w$ is not a source. So $w$ is not contained in $X$.
Moreover since $h_c(v)=1 \mod 3$ and $X$ only contains vertices $y$
satisfying $h_c(y) \in \{0,2\} \mod 3$, we have $v \notin X$. Thus
$vwz$ is a certifying backward path in $D_c\setminus (X\cup
B_c)$. \fpro

%

From now on, we consider a nice coloring of $G$ with maximal height
among all the nice colorings of $G$. This choice implies that if we
perform a switch coloring in $c$, case {\bf (c)} of
Claim~\ref{claim:switch} cannot occur.  
The next claim establishes some properties on the structure of
the last levels of $D_c$ for such nice colorings of $G$.

\begin{claim}
\label{claim:completebipartite}
Let $c$ be a nice coloring of $G$ with maximal height. 
Then the graph $G$ induces a complete bipartite
graph on $V^c_{h(c)}\cup B_c$ with bipartition $(V^c_{h(c)},B_c)$ and
also on $V^c_{h(c)}\cup V^c_{h(c)-1}$ with bipartition $(V^c_{h(c)},
V^c_{h(c)-1})$.
\end{claim}

\pro We know that $B_c$, $V^c_{h(c)}$ and $V^c_{h(c)-1}$ form three
independent sets of $G$. By maximality of $h(c)$, for every vertex $b$
of $B_c$ a switch recoloring on $b$ produces case {\bf (d)} of
Claim~\ref{claim:switch}, so we have $V^c_{h(c)}\subseteq
N^-_{D_c}(b)$. Thus $G$ induces a complete bipartite graph on
$V^c_{h(c)}\cup B_c$ with bipartition $(V^c_{h(c)},B_c)$.

Now let us prove that $G$ induces a complete bipartite graph on
$V^c_{h(c)}\cup V^c_{h(c)-1}$ with bipartition $(V^c_{h(c)},
V^c_{h(c)-1})$. By contradiction assume that there exist vertices
$x\in V^c_{h(c)}$ and $y\in V^c_{h(c)-1}$ such that $\{x,y\}\notin
E(G)$. First we will prove that $y$ has an in-neighbour in
$V^c_{h(c)}$. Indeed, we consider a vertex $b$ in $B_c$ and remark
that $y$ has no out-neighbour which is also an in-neighbour of $b$.
Otherwise let $y'$ be such a vertex and denote by $i$ its level,
ie. $y'\in V^c_i$.  As $yy'\in E(D_c)$ we would have $h(c)-1=i-1\mod
3$ and as $y'b\in E(D_c)$ we would have $i=2-1\mod 3$. So we would get
$h(c)=2\mod 3$ a contradiction to $h(c)=0\mod 3$, previously noticed.
So the only neighbours of $y$ which are in-neighbours of $b$ in $D_c$
are in-neighbours of $y$ and lie in $V_{h(c)}^c$.  Now we apply a
switch recoloring on $b$ to obtain the coloring $c'$. By assumption
$G$ is not a counter-example to Theorem~\ref{theo:3colorful}, and then
$B_{c'}$ contains at least one vertex $z$. By
Claim~\ref{claim:switch}~{\bf (e)}, we know that $B_{c'}\subseteq
N^-_{D_c}(b)$.  Moreover, Claim~\ref{claim:switch}~{\bf (b)} ensures
that $V_{h(c')}^{c'}=V_{h(c)-1}^{c}$ (because $h(c)=0\mod 3$). Thus
$y$ belongs to $V_{h(c')}^{c'}$, and by maximality of $h(c)=h(c')$ the
first part of the claim implies that $yz$ is an arc of $D_{c'}$.  Thus
$z$ was an in-neighbour of $b$ in $D_c$ and is a neighbour of $y$. By
the previous remark we know that $z\in V_{h(c)}^c$ and $z$ is an
in-neighbour of $y$ in $D_c$.\\
%
%
Now, in $D_c$, we know that $y$ has at least one non neighbour in
$V_{h(c)}^{c}$ (the vertex $x$) and also at least one in-neighbour in
$V_{h(c)}^{c}$.  We denote by $Y$ the set $N^-_{D_c}(y)$, which is
contained in $V_{h(c)}^c$, and apply an initial recoloring on the
initial section $\{y\} \cup Y$ to obtain the coloring $c'$.  Let us
prove that every vertex of $G$ has a certifying path in $D_{c'}$. First,
every vertex of $V_{h(c)}^{c}\cup \{y\}\cup B_c =
(V_{h(c)}^{c}\setminus Y) \cup B_c \cup Y \cup \{y\}$ has a certifying
path.  Indeed for every vertex $z\in V_{h(c)}^{c}\setminus Y$ (which
is non empty since it contains $x$), for every vertex $b\in B_c$ and
for every vertex $z'$ of $Y$ (which is non empty by the previous
paragraph), the oriented path $zbz'y$ exists in $D_{c'}$. Moreover,
the oriented graph $D_{c}\setminus ( Y\cup \{y\})$ is unchanged by the
recoloring. So it is possible to apply Claim~\ref{claim:technical}
with $X=Y\cup \{y\}$ to conclude that every vertex of $D_{c}\setminus (
Y\cup \{y\} \cup B_c)$ has also a certifying path in $D_{c}\setminus (
Y\cup \{y\})$=$D_{c'}\setminus ( Y\cup \{y\})$.\\ In all, every vertex
of $G$ has a certifying path in $D_{c'}$, a contradiction to the fact that
$G$ is a counter-example to Theorem~\ref{theo:3colorful}.
\fpro

The final claim gives more precision on the in-neighbourhood of the
vertices of $B_c$ in $D_c$.

\begin{claim}
\label{claim:onlylastlevel}
For a coloring $c$ of $G$ with maximal height, every vertex $b$ of
$B_c$ satisfies $N_{D_c}^-(b)= V^c_{h(c)}$.
\end{claim}

\pro Let $b$ be a vertex of $B_c$. By
Claim~\ref{claim:completebipartite} we know that $V_{h(c)}^c\subseteq
N^-_G(b)$. So assume that $b$ has an in-neighbour $u \notin
V_{h(c)}^c$, and consider the initial recoloring of the initial
section $V_{h(c)}^c \cup V_{h(c)-1}^c$ of $D_c$. We denote by $c'$ the
obtained coloring of $G$, and let us prove that there is a certifying
path in $c'$ for every vertex of the graph $G$. Let $z,z',b'$ be
respectively vertices of $V_{h(c)}^c$, $V_{h(c)-1}^c$ and $B_c$. By
Claim~\ref{claim:completebipartite} both $(V_{h(c)}^c,V_{h(c)-1}^c)$
and $(B_c,V_{h(c)}^c)$ induce complete bipartite graphs in $G$, and
then $b'zz'$ is an oriented path in $D_{c'}$. So, it is a backward
certifying path for $z'$ and a forward certifying path of
$b'$. Moreover, $ubz$ is also an oriented path of $D_{c'}$, and then a
certifying backward path for $z$. So every vertex of $B_c\cup
V_{h(c)}^c\cup V_{h(c)-1}^c$ has a certifying path in $D_{c'}$.  To
conclude we notice that the oriented graph $D_c\setminus (V_{h(c)}^c
\cup V_{h(c)-1}^c)$ is unchanged by the recoloring.  So we apply
Claim~\ref{claim:technical} with $X=V_{h(c)}^c \cup V_{h(c)-1}^c$
and conclude that every vertex of $D_c\setminus (V_{h(c)}^c
\cup V_{h(c)-1}^c\cup B_c)$ has a certifying path in $D_c\setminus (V_{h(c)}^c
\cup V_{h(c)-1}^c)=D_{c'}\setminus (V_{h(c)}^c
\cup V_{h(c)-1}^c)$.\\
In all, every vertex of $G$ has
a certifying path in $D_{c'}$, a contradiction to the fact that $G$ is a
counter-example to Theorem~\ref{theo:3colorful}.
\fpro

Now, it is possible to conclude the proof of
Theorem~\ref{theo:3colorful} by applying repeated switch colorings on
vertices of $B_c$. More precisely, let $c$ be a coloring of $G$ with
maximal height. By Claim~\ref{claim:onlylastlevel}, the
in-neighbourhood of every vertex $b$ of $B_c$ is exactly $V^c_{h(c)}$,
and its out-neighbourhood is $\{r_c\}$. So, as $G$ has no twins by
Claim~\ref{claim:no-twin}, it means that $B_c$ contains only one
vertex which is linked to $r_c$ and to all the vertices of
$V_{h(c)}^c$. Then, we consider the following partition of $G$: ${\cal
  U}_c=(U^c_1,\dots ,U^c_{h(c)+1})=(V_1^c, B_c, V^c_{h(c)},
V_{h(c)-1}^c,\dots , V_3^c, V_2^c\setminus B^c)$. By the previous
remark, we have $|U^c_1|=|U^c_2|=1$ and the neighbourhood of the
unique vertex of $U^c_2$ is included in $U^c_1\cup U^c_3$. Now, if we
apply a switch coloring on the unique vertex of $B_c$ and obtain a
coloring $c'$, properties {\bf (b)} and {\bf (d)} of
Claim~\ref{claim:switch} imply that $U^{c'}_i=U^{c}_{i+1}$ for
$i=1,\dots ,h(c)$ and $U^{c'}_{h(c)+1}=U^{c}_{1}$. As $c'$ is also a
nice coloring of $G$ with maximal height, we also have that
$|U^{c'}_2|=|U^c_3|=1$ and that the neighbourhood of the unique vertex
of $U^{c'}_2=U^c_3$ is $U^{c'}_1\cup U^{c'}_3=U^c_4\cup U^c_2$. By
repeating switch colorings, a direct induction shows that, for every
$i$, $U_i^{c}$ contains exactly one vertex, and the neighbourhood of
this vertex is $U_{i-1}^c\cup U_{i+1}^c$. So, the graph $G$ is a
cycle, a contradiction to Lemma~\ref{lem:cycleok}.

\section{Concluding remarks}
\label{sec:conclusion}

Notice that we can derive a polynomial time algorithm from the
proof of Theorem~\ref{theo:3colorful}.  Indeed, one can verify that
all the proofs of the lemmas and the claims provide algorithms in
order to improve $B^c$ at each step or find a coloring for which
every vertex admits a colorful path. So, given a 3-chromatic graph
$G$ different from $C_7$ and a $3$-coloring of $G$, we can find in
polynomial time a $3$-coloring of $G$ for which every vertex of $G$
is the beginning of a colorful path.

\medskip

Besides, it seems that the methods used in the proof of
Theorem~\ref{theo:3colorful} cannot be immediately generalized to
graphs with higher chromatic number. We can nevertheless
state the following weaker result for $4$-chromatic graphs.

\begin{lemma}\label{lem:C4}
 Every $4$-chromatic graph $G$ containing a cycle of length $4$ admits
 a $4$-coloring such that every vertex of $G$ is the beginning of a
 colorful path.
\end{lemma}

\begin{proof}
Let $c$ be a $4$-coloring of a 4-chromatic graph $G$. Denote by
$H=x_1x_2x_3x_4x_1$ a (non necessarily induced) cycle of length 4 of
$G$.  Recall that Lemma~\ref{lem:reachable-set} ensures that the
conclusion holds if $D_c$ contains a circuit. If the four colors
appear on $H$, then up to a permutation on colors, we can assume that
$c(x_i)=i$ for $i=1,2,3,4$. So, $H$ appear as an oriented cycle in
$D_c$ and we are done. So we can assume that $H$ is colored with two
or three colors. Let us prove that in both cases, either there is a
oriented cycle in $D_c$ or the number of colors appearing on $H$ in a
4-coloring of $G$ can be increased.

 First assume that only two colors of $c$ appear in $H$. Up to a
 permutation on colors, we can assume that $x_1,x_3$ are colored
 with $1$ and that $x_2,x_4$ are colored with $2$. For a vertex $y$
 of $G$, we denote by $S_{D_c}^-(y)$ the set of vertices of $G$ which
 are the beginning of an oriented path in $D_c$ with end $y$. Notice
 that $S_{D_c}^-(y)$ is always an initial section of $D_c$. Consider
 the set $S_{D_c}^-(x_1)$. If it contains $x_2$ or $x_4$, it means
 that $D_c$ has an oriented cycle. So, we assume that $x_2\notin
 S_{D_c}^-(x_1)$ and $x_4\notin S_{D_c}^-(x_1)$. If $x_3\notin
 S_{D_c}^-(x_1)$, then we perform an initial recoloring on
 $S_{D_c}^-(x_1)$ and in the resulting coloring, $H$ receives three
 colors.  Indeed $x_2$, $x_3$ and $x_4$ do not belong to
 $S_{D_c}^-(x_1)$ and still have the same color (i.e. respectively 2,
 1 and 2), and $x_1$ receive now color 4. Otherwise there is a
 directed path in $D_c$ from $x_3 $ to $x_1$. But, by symmetry there
 also exists a directed path from $x_1$ to $x_3$, which provides an
 oriented cycle in $D_c$.

Assume now that $H$ receives three colors by $c$. Up to permutation on
colors, we can assume that $x_1$ is colored with $1$, $x_2$ and $x_4$
are colored with $2$ and $x_3$ is colored with $3$.  Consider the set
$S_{D_c}^-(x_2)$. If it contains $x_3$, then $D_c$ has an oriented
cycle. So we assume that $x_3\notin S_{D_c}^-(x_2)$. If
$S_{D_c}^-(x_2)$ does not contain $x_4$, then we apply an initial
recoloring on it. As $x_1\in S_{D_c}^-(x_2)$, $x_1$, $x_2$, $x_3$ and
$x_4$ respectively receive colors 4, 1, 3, 2, and we have a coloring
of $G$ with four different colors on $H$. So, assume that $x_4\in
S_{D_c}^-(x_2)$. It means that there exists an oriented path from
$x_4$ to $x_2$ in $D_c$. By symmetry, there exists also a directed
path from $x_2$ to $x_4$ in $D_c$ and $D_c$ contains an oriented
cycle.
\end{proof}

In particular, if a $3$-chromatic graph contains a cycle of length
three, it appears as an oriented cycle in $D_c$ for any 3-coloring
$c$, and then Lemma~\ref{lem:reachable-set} ensures that
Conjecture~\ref{conj:colorful} holds. If a $4$-chromatic graph
contains a cycle of length four, then Lemma~\ref{lem:C4} ensures that
Conjecture~\ref{conj:colorful} holds. In both cases the proofs are
simple. It raises the following  natural question.

\begin{problem} 
Does Conjecture~\ref{conj:colorful} hold for $k$-chromatic connected
graphs containing a cycle of length $k$?
\end{problem}

To conclude, note also that in the case of $2$ and $3$-chromatic
graphs, every colorful path is either forward or backward (recall
that forward (resp.  backward) means that the color of the $i$-th
vertex is the color of the $(i-1)$-th vertex plus (resp. minus)
one). When $D_c$ contains an oriented cycle or in Lemma~\ref{lem:C4},
we obtain certifying paths which are all forward or
backward. In~\cite{ATT11}, M. Alishahi, A. Taherkhani and C. Thomassen provides
paths with $\chi-1$ vertices intersecting $\chi-1$ colors which are
union of at most $2$ increasing paths. It raises the following
strengthened conjecture.

\begin{conjecture}
Every $k$-chromatic connected graph different from $C_7$ admits a
$k$-coloring such that every vertex is the end of a forward or
backward certifying path.
\end{conjecture}

\end{document}